\documentclass[12pt]{article}      
\usepackage{amsmath,amssymb,amsthm, amscd, graphicx, verbatim, setspace} 
\usepackage{setspace}

\theoremstyle{plain}
\newtheorem{thm}{Theorem}
\newtheorem{lem}[thm]{Lemma}
\newtheorem{cor}[thm]{Corollary}

\title{Linear bounds on matrix extremal functions using visibility hypergraphs}
\date{}
\author{Jesse Geneson\thanks{Supported by the NSF Graduate Research Fellowship under Grant No. 1122374.} \qquad Lilly Shen\\
\small Department of Mathematics\\[-0.8ex]
\small MIT\\[-0.8ex]
\small Massachusetts, U.S.A.\\
\small\tt geneson@math.mit.edu\\
\small\tt lshen01@college.harvard.edu
}
\begin{document}
\maketitle

\begin{abstract}
The $0-1$ matrix $A$ contains a $0-1$ matrix $M$ if some submatrix of $A$ can be transformed into $M$ by changing some ones to zeroes. If $A$ does not contain $M$, then $A$ avoids $M$. Let $ex(n, M)$ be the maximum number of ones in an $n \times n$ $0-1$ matrix that avoids $M$, and let $ex_{k}(m, M)$ be the maximum number of columns in a $0-1$ matrix with $m$ rows that avoids $M$ and has at least $k$ ones in every column. A method for bounding $ex(n, M)$ by using bounds on the maximum number of edges in bar visibility graphs was introduced in (R. Fulek, Discrete Mathematics 309, 2009). By using a similar method with bar visibility hypergraphs, we obtain linear bounds on the extremal functions of other forbidden $0-1$ matrices.
\end{abstract}

\section{Introduction}
For a collection $S$ of $0-1$ matrices, let the \emph{weight extremal function} $exs(n, S)$ denote the maximum number of ones in an $n \times n$ $0-1$ matrix which avoids every matrix in $S$, and let the \emph{column extremal function} $exs_{k}(m, S)$ denote the maximum number of columns in a $0-1$ matrix with $m$ rows which avoids every matrix in $S$ and has at least $k$ ones in every column.

The $0-1$ matrix weight extremal function was used to bound the complexity of an algorithm for minimizing rectilinear path distance while avoiding obstacles \cite{Mit}. Weight extremal functions of $0-1$ matrices have also been used to bound the maximum number of unit distances in a convex $n$-gon \cite{Fure}, the number of edges in ordered graphs on $n$ vertices avoiding fixed ordered subgraphs \cite{KM, PT}, and the maximum length of sequences which avoid fixed subsequences \cite{Pet}.

The column extremal function of matrices is the analogue of a sequence extremal function, the maximum number of letters in a sequence with $m$ contiguous blocks of distinct letters that avoids a fixed subsequence (or collection of subsequences) and has at least $k$ occurrences of each letter. This sequence extremal function was bounded for alternations and collections of sequences called $(r, s)$-formations in \cite{Niv}. The column extremal functions for collections of matrices called doubled $(r, s)$-formations were bounded in \cite{CK}. Both the weight and the column extremal functions of permutation matrices were used in \cite{Fox} to improve upper bounds on the Stanley-Wilf limits of forbidden permutations. 

Since an $n \times n$ matrix has $n^{2}$ entries, then $exs(n, S) \leq n^{2}$ for every collection $S$ of $0-1$ matrices. If $S$ contains only one matrix $M$, then we write $exs(n, S) = ex(n, M)$ and $exs_{k}(m, S) = ex_{k}(m, M)$. If $M$ has at least two ones, then $ex(n, M) \geq n$ since a matrix with ones only in a single column or a matrix with ones only in a single row avoids $M$.

Furedi and Hajnal asked for a characterization of all $0-1$ matrices $P$ such that $ex(n, P) = O(n)$ \cite{FH}. A corresponding problem for the column extremal function is to characterize all $0-1$ matrices $P$ with $j$ rows such that for every fixed $k \geq j$, $ex_{k}(m, P) = O(m)$. Another related problem is to characterize all $0-1$ matrices $P$ for which $ex_{k}(m, P) = O(\frac{m}{k})$.

Fulek \cite{Ful} found bounds on the extremal functions of the patterns $L_{1}$ and $L_{2}$ (Figure \ref{Lpatt}) using visibility representations constructed by treating rows of a given $n \times n$ matrix as vertices and projections of ones on lower rows as edges. He bounded the number of ones in the matrix by limiting the multiplicity of edges in the visibility representation based on the forbidden $0-1$ matrices.

\begin{figure}[htbp]
\centering

\begin{tabular}{cc}
$\begin{pmatrix}
0 & 1 & 1 & 0 \\
1 & 0 & 0 & 1 \\
0 & 1 & 0 & 0 
\end{pmatrix}$
&
$\begin{pmatrix}
0 & 1 & 1 & 1 & 0\\
1 & 0 & 0 & 0 & 1\\
0 & 0 & 1 & 0 & 0
\end{pmatrix}$ \\
\end{tabular}
\caption{$L_1$ and $L_{2}$}\label{Lpatt}
\end{figure}

In Section \ref{prelim} we prove some general facts about $exs(n, S)$ and $exs_{k}(m, S)$. In Section \ref{bar} we define bar visibility hypergraphs, bound their number of edges, and use these bounds to show that $ex(n, L_{3}) = O(n)$ and $ex_{k}(m, L_{3}) = O(\frac{m}{k})$ (Figure \ref{Lpatt1}). We also bound the extremal functions of forbidden collections of $0-1$ matrices corresponding to bar visibility hypergraphs.

\begin{figure}[htbp]
\centering

\begin{tabular}{cc}
$\begin{pmatrix}
0 & 1 & 1 & 1 & 0\\
1 & 0 & 0 & 0 & 0\\
0 & 0 & 0 & 0 & 1\\
0 & 0 & 1 & 0 & 0
\end{pmatrix}$ \\
\end{tabular}
\caption{$L_3$}\label{Lpatt1}
\end{figure}

\section{Facts about $exs(n, S)$ and $exs_{k}(m, S)$}\label{prelim}

If $k > m$, then $exs_{k}(m, S) = 0$ for any collection $S$. If $k \leq m$ and every matrix in $S$ has at least $k$ rows with ones, then $exs_{j}(m, S) = \infty$ for $j < k$ since a matrix with $m$ rows, any number of columns, and ones in every entry in the first $j$ rows avoids every matrix in $S$. 

If some matrix in $S$ has $k$ rows and $c$ columns, then $exs_{j}(m, S) \leq (c-1) {m \choose k}$ for all $j \geq k$. This is because a matrix $A$ with $(c-1) {m \choose k}+1$ columns, in which every column contains at least $j$ ones, must contain $c$ columns with ones in the same $k$ rows by the pigeonhole principle, so $A$ contains the matrix in $S$ with $k$ rows and $c$ columns.

Let $ex(m,n,P)$ (resp. $exs(m,n,S)$) be the maximum weight of an $m \times n$ matrix which avoids the $0-1$ matrix $P$ (resp. collection $S$), so that $ex(n, P) = ex(n, n, P)$. For every column of $P$, draw a segment connecting the topmost and bottommost ones. Call $P$ range-overlapping if, for every pair of columns of $P$, there exists a horizontal line passing through the corresponding segments of both columns.

\begin{figure}[htbp]
\begin{center}
\includegraphics[scale=.3]{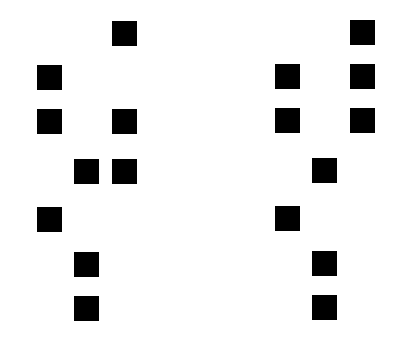}
\caption{The pattern on the left is range-overlapping. The pattern on the right is not range-overlapping because its final two columns have disjoint ranges.}
\label{rangeoverlap1}
\end{center}
\end{figure}

\begin{thm}
\label{rangeo}
For any range-overlapping $P$, $ex(m,n,P) \le k(ex_k(m,P)+n)$.
\end{thm}  

\begin{proof}
Let $A$ be a $0-1$ matrix with $m$ rows, $n$ columns, and $ex(m,n,P)$ ones which avoids $P$.  Consider column $c$ in $A$. Starting from the top of $c$, divide its ones into clusters of size $k$, deleting the at most $k-1$ remaining ones in the column. Move each cluster $g$ after the first cluster horizontally to a new column immediately to the right of the column containing the cluster which was above $g$ in the original column. Delete any columns with no ones. Call the newly formed matrix $A'$, and suppose that $A'$ has $n'$ columns.

Suppose for contradiction that $A'$ contains $P$, and consider two cases. If the copy of $P$ in $A'$ contained at least two columns that originated from the same column $c$, then $P$ could not be range-overlapping because the construction separated columns of $A$ into clusters with non-range-overlapping row indices to obtain $A'$. If the copy of $P$ in $A'$ only contained columns derived from different original columns, then the original matrix $A$ also contained $P$, a contradiction.

Therefore, $A'$ cannot contain $P$.  In our construction, we deleted a maximum of $n(k-1)$ ones. Each column of $A'$ contains $k$ ones, so $ex(m,n,P) \le k(n'+n) \le k(ex_k(m,P)+n)$.
\end{proof}

\begin{lem}\label{reverse}
If $S$ is a collection of $0-1$ matrices, $c$ is a constant, and $g$ is a function such that $exs(m, n, S) \leq g(m)+c n$ for all $m$ and $n$, then $exs_{k}(m, S) \leq \frac{g(m)}{k-c}$ for $k > c$.
\end{lem}

\begin{proof}
Fix $k > c$. Any matrix with $m$ rows, $exs_{k}(m, S)$ columns, and $k$ ones per column which avoids $S$ has at most $exs(m, exs_{k}(m, S), S) \leq g(m) + exs_{k}(m, S) c$ ones. Therefore $exs_{k}(m, S) k \leq g(m)+ exs_{k}(m, S) c$, so $exs_{k}(m, S) \leq \frac{g(m)}{k-c}$ for all $m$.
\end{proof}

\begin{cor}\label{linear}
If $S$ is a collection of $0-1$ matrices such that $exs(n, S) = O(n)$ and every matrix in $S$ contains the $2 \times 2$ identity matrix or two ones in the same row, then $exs_{k}(m, S) = \theta(\frac{m}{k})$. 
\end{cor}

\begin{proof}
Since $exs(m, n, S)  \leq max\left\{exs(m, S), exs(n, S)\right\}$, the upper bound follows from Lemma \ref{reverse}. For the lower bound, let $K$ be the matrix with $m$ rows and $\lfloor \frac{m}{k} \rfloor$ columns obtained from an $\lfloor \frac{m}{k} \rfloor \times \lfloor \frac{m}{k} \rfloor$ identity matrix by replacing every row $r$ in the identity matrix with the $k \times \lfloor \frac{m}{k} \rfloor$ submatrix consisting of $k$ adjacent copies of $r$ and then adding rows full of zeroes at the bottom until $K$ has $m$ rows. Let $K'$ be obtained from $K$ by reflecting over a vertical line. Then $exs_{k}(m, S) \geq \frac{m}{k}-1$ since $K'$ avoids $S$.
\end{proof}

Marcus and Tardos \cite{MT} showed that $ex(n, P) = O(n)$ for any permutation matrix $P$, solving the Stanley-Wilf conjecture. Call a $0-1$ matrix \emph{light} if it has no pair of ones in the same column. The Marcus Tardos theorem was generalized in \cite{G} to show that $ex(n, Q) = O(n)$ for any light $0-1$ matrix $Q$ such that for every two columns $c_{0}$ and $c_{1}$ in $Q$ with ones in the same row $r$, all columns between $c_{0}$ and $c_{1}$ also contain ones in row $r$. Corollary \ref{linear} can be applied to the matrices $Q$ to show that $ex_{k}(m, Q) = \theta(\frac{m}{k})$, as well as to the matrices $L_{1}$ and $L_{2}$ in \cite{Ful}.

Let $P_{r,c}$ denote the $r \times c$ matrix filled with ones. We find the exact value of $ex_k(m,P_{k,c})$ using a simple counting argument, and then bound $ex_k(m,P_{r,c})$ for all $k \geq r$. 

\begin{lem}
$ex_{k}(m,P_{k,c})= (c-1){m \choose k}$.
\end{lem}

\begin{proof}
There exist ${m \choose k}$ possible configurations of $k$ ones in columns of height $m$.  By the pigeonhole principle, a matrix with $(c-1){m \choose k}+1$ columns must have a configuration of $k$ ones that occurs in at least $c$ columns. Therefore, $ex_k(m,P_{k,c}) \le (c-1){m \choose k}$.  We can obtain the same lower bound by constructing a matrix with every possible configuration $c-1$ times; this matrix avoids $P_{k,c}$. Hence, $ex_k(m,P_{k,c}) = (c-1){m \choose k}$.
\end{proof}

\begin{thm}
For all fixed $k \geq r$, $ex_k(m,P_{r,2}) = \theta(m^r)$.
\end{thm}

\begin{proof}
First observe that $ex_r(m,P_{r,2}) = O(m^r)$, according to the last lemma. Since $ex_k(m, P)$ is decreasing in $k$ for every $P$, $ex_k(m,P_{r,2}) = O(m^r)$ for $k \geq r$.

For any $0-1$ matrix $M$, let $G_{M}$ be the graph obtained from $M$ by letting every column of $M$ be a vertex and adding an edge between two vertices if and only if their corresponding columns have ones in exactly $r-1$ common rows. Note that $r-1$ is the maximum number of rows a pair of columns may share without containing $P_{r,2}$. We proceed by induction on $k$ to prove that $ex_k(m,P_{r,2}) = \Omega(m^r)$. 

The first case is $k = r$. The last lemma implies that $ex_r(m,P_{r,2}) = {m \choose r}$. Observe that if $M$ is a $0-1$ matrix with $m$ rows and $r$ ones in each column which avoids $P_{r, 2}$, then the maximum degree of any vertex of $G_{M}$, $\Delta(G_{M})$, is at most $r(m-r)$.

Fix $k > r$.  Let $M$ be the matrix obtained in the last inductive step avoiding $P_{r,2}$ with $m' \leq xm$ rows, ${m \choose r}$ columns, and $k-1$ ones in each column, such that $\Delta(G_{M}) \leq ym$ for some constants $x$ and $y$ that depend only on $r$ and $k$. 

Using a greedy algorithm, color the vertices in $G_{M}$ using $\Delta(G_{M})+1$ colors so that no two vertices with a common edge have the same color. Let $C: V_{G_{M}} \rightarrow \left\{1, \ldots, \Delta(G_{M})+1 \right\}$ be the coloring function. Construct a new matrix $M'$ from $M$ by adding $\Delta(G_{M})+1$ new rows to the bottom of $M$ with a one in row $m'+r$ of column $b$ if and only if $C(b)=r$. Since $\Delta(G_{M}) \leq ym$, the resulting matrix $M'$ has at most $xm+ ym +1$ rows.

Fix a column $b$ of $M'$. A column $c$ is a neighbor of $b$ in $G_{M'}$ but not in $G_{M}$ only if there are exactly $r-2$ rows in $M$ which have ones in both $b$ and $c$, and $b$ and $c$ were assigned the same number in the coloring of $G_{M}$. Since $b$ has $k-1$ rows with ones in $M$, there are ${k-1 \choose r-2}$ ways to choose $r-2$ rows that contain ones in column $b$ of $M$. For every set of $r-2$ rows in which column $b$ has ones in $M$, there are at most $\frac{xm-(k-1)}{(k-1)-(r-2)}$ possible new neighbors $c$ of $b$ in $M'$ with ones in those rows, since every pair of new neighbors of $b$ in $G_{M'}$ with ones in those $r-2$ rows get the same color in $G_{M}$ and thus have no common rows in $M$ containing ones besides the $r-2$ rows each neighbor shares with $b$. Then $\Delta(G_{M'}) \leq \frac{xm-(k-1)}{(k-1)-(r-2)}{k-1 \choose r-2}+ym$, completing the induction.
\end{proof}

We proved that for each $k \geq r$, there exists a constant $a_{r,k}$ such that $ex_k(a_{r,k}m, P_{r,2}) \geq {m \choose r}$. To generalize the bounds for all integral values of $m$, let $n$ satisfy $a_{r,k}n \le m \le a_{r,k}(n+1)$ and observe that $ex_k(a_{r,k} n, P_{r,2}) \geq {n \choose r} \geq {\frac{m}{a_{r,k}}-1 \choose r}$. Hence $ex_k(m, P_{r,2}) = \theta(m^r)$ for all fixed $k \geq r$.

\begin{cor}
For all fixed $k \geq r$ and $c \geq 2$, $ex_k(m,P_{r,c}) = \theta(m^r)$.
\end{cor}

\begin{proof}
For $c \ge 2$, $ex_k(m,P_{r,c}) \ge ex_k(m,P_{r,2})$ since $P_{r,c}$ contains $P_{r, 2}$. Therfore $ex_k(m,P_{r,c})= \Omega(m^r)$ for all $k \ge r$. Since $ex_{r}(m,P_{r,c})=(c-1){m \choose r}$, then $ex_{k}(m, P_{r,c}) = O(m^r)$ for all $k \ge r$. This shows that $ex_{k}(m,P_{r,c})= \theta(m^r)$ for all fixed $k \ge r$ and $c \ge 2$.    
\end{proof}

\begin{cor}
If $P$ is a $0-1$ matrix with $r$ rows which contains $P_{r,2}$, then $ex_{k}(m, P) = \theta(m^r)$ for all fixed $k \ge r$.
\end{cor}

\section{Bar $s$-visibility hypergraphs and $0-1$ matrices}\label{bar}

A bar $s$-visibility hypergraph's vertices are finite, disjoint, horizontal bars in the plane and its edges are subsets of vertices of size $s+2$ for which a vertical segment exists which intersects only the vertices in the edge. Using the fact that every bar visibility graph (i.e. bar $0$-visibility hypergraph) is a planar graph, Fulek converted $0-1$ matrices into bar $0$-visibility hypergraphs to prove linear bounds on the extremal functions of patterns $L_{1}$ and $L_{2}$ \cite{Ful}.

For $r, s \geq 0$, define $T_{r,s}$ to be the collection of matrices $M$ which have $r+s+2$ rows and $r+2s+2$ columns such that $M$ restricted to the first $s+1$ columns and rows $2, \ldots, s+2$ is a $(s+1)\times (s+1)$ permutation matrix, $M$ restricted to the last $s+1$ columns and rows $2, \ldots, s+2$ is a $(s+1) \times (s+1)$ permutation matrix, $M$ restricted to the middle $r$ columns and the last $r$ rows is an $r \times r$ permutation matrix, and $M$ has ones in the middle $r$ columns in row $1$. For example $T_{1,0}$ contains a single $3 \times 3$ matrix with four ones in a diamond formation. One of the patterns in $T_{4,1}$ appears in Figure \ref{trs}, with black squares representing ones and white squares representing zeroes.

We extend the method from \cite{Ful} to show that $exs(n, T_{r,s}) = O(n)$ and $exs_{k}(m, T_{r,s}) = \theta(\frac{m}{k})$ for all $r, s \geq 0$. The first step is to prove a linear bound on the number of edges in a bar $s$-visibility hypergraph with $n$ vertices. This proof is similar to the proof of the maximum number of edges in bar $s$-visibility graphs with $n$ vertices in \cite{Dean}. We assume that all bar endpoints have distinct coordinates since this does not decrease the maximum number of edges.

\begin{figure}[htbp]
\begin{center}
\includegraphics[scale=.5]{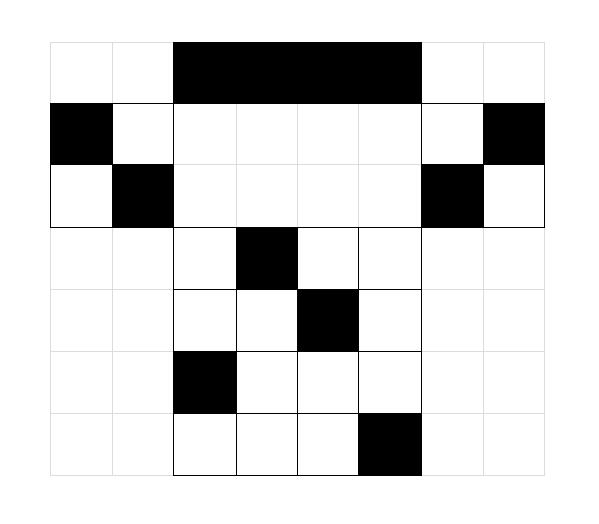}
\caption{A matrix in $T_{4,1}$}
\label{trs}
\end{center}
\end{figure}

\begin{lem}
All bar $s$-visibility hypergraphs with $n$ vertices have at most $(2s+3)n$ edges.
\end{lem}

\begin{proof}
Scan any representation of the given bar $s$-visibility hypergraph in the plane from left to right, making a list of distinct edges. Add an edge to the list whenever the scan for the first time shows part of the representation in which a vertical segment can be drawn which intersects each of the vertices in the edge. Then edges will only be added to the list whenever the scan passes the left or right end of some bar.

For each bar $B$, the maximum possible number of edges added to the list when the scan passes the left end of $B$ is $s+2$ since there are at most $s+2$ vertical segments representing different edges which pass through the left end of $B$ and through $s+1$ other bars. The maximum possible number of edges added to the list when the scan passes the right endpoint of $B$ is $s+1$, since at the right endpoint of $B$ there are at most $s+1$ vertical segments representing different edges which pass through $s+2$ bars other than $B$, at least one of which is below $B$ and at least one of which is above $B$. Then there are at most $(2s+3)n$ edges on the list since the representation contains $n$ bars.
\end{proof}

In the next lemma, we change $0-1$ matrices avoiding $T_{r,s}$ into bar $s$-visibility hypergraphs, and then show that the resulting hypergraphs have edge multiplicity at most $r-1$.

\begin{thm}
\label{kvis}
For all $r \geq 0$ and $s \geq 0$, $exs(n, T_{r,s}) = \theta(n)$.
\end{thm}

\begin{proof}

Let $M$ be an $n \times n$ matrix which avoids $T_{r,s}$. Define $M'$ to be the matrix obtained from $M$ by deleting the first $s+1$ and last $s+1$ ones in every row, and the last $r$ ones in every column. Construct a representation of a bar $s$-visibility hypergraph $H$ from $M'$ by drawing a bar in each row that contains a one in $M'$ with left end at the first one of $M'$ in the row and right end at the last one of $M'$ in the row. For each one in $M'$ with at least $s+1$ ones below it in $M'$, draw a vertical line segment representing an edge starting from the one and extending through $s$ bars until reaching the $(s+1)^{st}$ bar below the one.

Suppose for contradiction that $H$ contains some edge $e$ with multiplicity at least $r$. Let $u_{1}, \ldots, u_{s+2}$, in order from top to bottom, be the rows of $M'$ which contain the vertices in the edge $e$, and let $c_{1}, \ldots, c_{r}$ be columns of $M'$ which contain $r$ vertical segments representing the copies of $e$. Let $v_{1}, \ldots, v_{r}$ be distinct rows of $M$ such that $v_{i}$ contains one of the bottommost $r$ ones of $c_{i}$ for each $i = 1, \ldots, r$; let $d_{1}, \ldots, d_{s+1}$ be distinct columns of $M$ such that $d_{i}$ contains one of the $s+1$ leftmost ones of $u_{i}$ for each $i = 2, \ldots, s+2$; and let $e_{1}, \ldots, e_{s+1}$ be distinct columns of $M$ such that $e_{i}$ contains one of the $s+1$ rightmost ones of $u_{i}$ for each $i = 2, \ldots, s+2$. Then the submatrix of $M$ consisting of rows $u_{1}, \ldots, u_{s+2}, v_{1}, \ldots, v_{r}$ and columns $c_{1}, \ldots, c_{r}, d_{1}, \ldots, d_{s+1}, e_{1}, \ldots, e_{s+1}$ contains a matrix in $T_{r,s}$, a contradiction.

Then every edge of $H$ has multiplicity less than $r$, so the number of ones in $M$ is at most $(3s+3+r)n+(r-1)(2s+3)(n-r)$.
\end{proof}

Observe that every element of $T_{r,s}$ contains the pattern $L_{3}$ for $r \geq 3$ and $s \geq 1$, which implies the following corollary.

\begin{cor}
$ex(n, L_{3}) = \theta(n)$.
\end{cor}

The next two corollaries follow from Corollary \ref{linear}. 

\begin{cor}
$exs_{k}(m, T_{r,s}) = \theta(\frac{m}{k})$.
\end{cor}

\begin{cor}
$ex_{k}(m, L_{3}) = \theta(\frac{m}{k})$.
\end{cor}

\section{Acknowledgments}
The authors thank the Research Science Institute, the MIT math department, Tanya Khovanova, Pavel Etingof, and Slava Gerovitch for supporting this research.


\begin{thebibliography}{}
\bibitem{CK} J. Cibulka, J. Kyncl, Tight bounds on the maximum size of a set of permutations with bounded VC-dimension. J. Comb. Theory, Ser. A 119(7): 1461-1478 (2012)
\bibitem{Dean} A. M. Dean, W. Evans, E. Gethner, J. D. Laison, M. A. Safari, W. T. Trotter, Bar k-visibility graphs. J. Graph Algorithms and Applications, 11(1): 45-59 (2007).
\bibitem{Fox} J. Fox, Stanley-Wilf limits are typically exponential. CoRR abs/1310.8378 (2013)
\bibitem{Ful} R. Fulek, Linear bound on extremal functions of some forbidden patterns in 0-1 matrices. Discrete Mathematics 309(6): 1736-1739 (2009).
\bibitem{Fure} Z. Furedi, The maximum number of unit distances in a convex n-gon, Journal of Combinatorial Theory Ser. A 55 (2), 316-320 (1990).
\bibitem{FH} Z. Furedi, P. Hajnal, Davenport-Schnizel theory of matrices, Discrete Mathematics, v.103 n.3, p.233-251 (1992)
\bibitem{G} J. Geneson, Extremal functions of forbidden double permutation matrices, Journal of Combinatorial Theory Series A, v.116 Issue 7, 1235-1244 (2009)
\bibitem{KM} M. Klazar, A. Marcus, Extensions of the linear bound in the Furedi-Hajnal conjecture, Adv. in Appl. Math. 38, no. 2, 258-266 (2006)
\bibitem{MT} A. Marcus, G. Tardos, Excluded permutation matrices and the Stanley-Wilf conjecture, Journal of Combinatorial Theory Series A, v.107 n.1, p.153-160 (2004)
\bibitem{Mit} J. Mitchell: Shortest rectilinear paths among obstacles, Department of Operations Research and Industrial Engineering Technical Report No. 739, Cornell University, Ithaca, New York (1987)
\bibitem{Niv} G. Nivasch, Improved bounds and new techniques for Davenport-Schinzel sequences and their generalizations, J. ACM, 57 (2010)
\bibitem{PT} J. Pach, G. Tardos, Forbidden paths and cycles in ordered graphs and matrices, Israel J. Math. 155, 309-334 (2006)
\bibitem{Pet} S. Pettie, Degrees of Nonlinearity in Forbidden 0-1 Matrix Problems, Discrete Mathematics 311: 2396-2410 (2011)
\end{thebibliography}
\end{document}